\documentclass[12pt]{amsart}

\usepackage{amsmath, color} \usepackage{amssymb} \usepackage{amsthm}
\usepackage{amscd}
\usepackage{tikz}
\usetikzlibrary{calc}
\usetikzlibrary{patterns}
\usetikzlibrary{matrix}
\usepgflibrary{arrows}
\usetikzlibrary{arrows}
\usetikzlibrary{decorations.pathreplacing}
\usetikzlibrary{decorations.pathmorphing}
\usepackage{epsfig}
\usepackage{enumerate}

\def\Tor{\operatorname{Tor}}

\DeclareMathOperator{\reg}{reg}

\theoremstyle{plain}
\newtheorem{theorem}{Theorem}[section]
\newtheorem{corollary}[theorem]{Corollary}
\newtheorem{proposition}[theorem]{Proposition}

\newtheorem{definition}[theorem]{Definition}

\newtheorem{lemma}[theorem]{Lemma}
\newtheorem{question}[theorem]{Question}

\newtheorem{observation}[theorem]{Observation}

\numberwithin{equation}{section}

\begin{document} 

\title{Powers of Edge Ideals of regularity three bipartite graphs}
\date{\today}

\author[Banerjee]{Ali Alilooee, Arindam Banerjee}
\address{Department of Mathematics, University of Virginia,
Charlottesville, VA, USA} \email{ab4cb@virginia.edu}
\address{Department of Mathematics, Dalhousie University,
Halifax, NS, Canada} \email{alilooee@mathstat.dal.ca}


\baselineskip 16pt \footskip = 32pt

\begin{abstract}
In this paper we prove that if $I(G)$ is a bipartite edge ideal with regularity three then for all $s\geq 2$ the regularity of $I(G)^s$ is exactly $2s+1$.
\end{abstract}

\maketitle \markboth{A. Alilooee, A.Banerjee}
{Regularity three bipartite graphs}

\section{Introduction}
\bigskip
  In this article we study the higher powers of edge ideals of regularity three bipartite graphs. 
  Previous studies have found classes of graphs, powers of whose edge ideals have linear minimal free resolution. 
  Herzog, Hibi and Zheng showed in~\cite{MR2091479} that powers of edge ideals with linear resolution have linear resolution themselves.  
  A classic result by Fr\"oberg~\cite{MR1171260}  says that an edge ideal has linear resolution if and only if there is no induced cycle of length greater than 
  three in its complement. 
  Francisco, Ha and Van Tuyl~\cite{MR2301246} showed that if any power of an edge ideal has linear minimal free resolution, then the 
  complement of the corresponding 
  graph has no induced four cycles, 
  which is equivalent to having a linear presentation due to~\cite{MR3011341}. 
  In light of these, and based on the C. Francisco Mcaulay 2 calculations, E. Nevo and I. Peeva  asked the following question, which is the base case of the Open Problem $1.11(2)$~\cite{MR3011341}.\\

\begin{question}\label{ques1}
Let $I(G)$ be the edge ideal of a graph $G$ which does not have any induced four cycle in its complement. If $\reg(I(G)) \leq 3$, then is it true that for all $s\geq 2$, $I(G) ^ s$ has linear minimal free resolution?
\end{question}
 One important fact about bipartite graphs is that the complement of a bipartite graph cannot have any induced cycle of length greater than four. 
  In light of Fr\"oberg's theorem and this fact, one can say that for bipartite graphs linear presentation implies linear resolution. 
   Due to these, we ask a question similar to Question 1.1 for bipartite graphs with a weaker hypothesis and answer it in the affirmative:
   
\begin{theorem}
Let $G$ be a bipartite connected graph with edge ideal $I(G)$. If $\text{reg}(I(G))=3$ then for all $s \geq 1$, $\text{reg}((I(G)^s)=2s+1$.
\end{theorem}
   
  For our proof, we use the combinatorial characterization of regularity three bipartite graphs by 
  Oscar Fern\'andez-Ramos and Phillippe Gilmenez proved in~\cite{MR3199032} and the techniques 
  introduced by the second author in~\cite{arindam2014}.
\bigskip

\section{Preliminaries}


Throughout this paper, we let $G$ be a finite simple graph with vertex set $V(G)$.
A subgraph $G' \subseteq G$ is called {\bf induced} if $uv$ is an edge of $G'$ whenever $u$ and $v$ are vertices of $G'$ and $uv$ is an edge of $G$.
The {\bf complement} of a graph $G$, for which we write $G^c$, 
is the graph on the same vertex set in which $uv$ is an edge of $G^c$ if and only if it is not an edge of $G$.  
Finally, let $C_k$ denote the cycle on $k$ vertices; a {\bf chord} is an edge which is not in the edge set of $C_k$. A cycle is called {\bf minimal} if it has no chord.  




\begin{definition}[Bipartite Graphs]
A graph $G$ is called {\bf bipartite} if there are two disjoint independent subsets $X,Y$ of $V(G)$ whose union is $V(G)$. Note that $X\subset V(G)$ is called 
{\bf independent} if there is 
no edge $e\in E(G)$ such that $e=xy$ for some $x,y\in X$.
\end{definition}
We have the following theorem for classifying bipartite graphs. For proof see~\cite{MR1367739}.
 \begin{theorem}[K\"onig, \cite{MR1367739}, Theorem $1.2.18$]\label{theorem:konig} 
A graph $G$ is bipartite if and only if $G$ contains no odd cycle. 
\end{theorem}
We also need the following definition for bipartite graphs. We take this definition from~\cite{MR3199032}.
\begin{definition}[Bipartite Complement]
 The {\bf bipartite complement} of a bipartite graph $G$ is a bipartite graph $G^{bc}$ over the same vertex set and same bipartition $V(G)=X\sqcup Y$ such that 
 $$E(G^{bc})=\{xy:x\in X,y\in Y, xy\notin G \}.$$
\end{definition}


 If $G$ is a graph without isolated vertices then let $S$ denote the polynomial ring on the vertices of $G$ over some fixed field $K$.  Recall that the \emph{edge ideal} of $G$ is 
\[
I(G) = (xy: xy \text{ is an edge of } G).
\]


\begin{definition}
Let $S$ be a standard graded polynomial ring over a field $K$. The {\bf Castelnuovo-Mumford regularity} of a finitely generated graded $S$ module $M$, 
written $\reg(M)$ is given by $$\reg(M):= \max \{j-i|\Tor_{i} (M,K)_j \neq 0 \}$$
\end{definition}

\begin{definition}
For every $s$ we say that $I(G)^s$ is {\bf $k$-steps linear} whenever the minimal free resolution of $I(G)^s$ over the polynomial ring
is linear for $k$ steps, i.e., $\Tor_{i}^S(I(G)^s,K)_j = 0$ for all $1\leq i\leq k$ and all $j\ne i+2s$. 

In particular we say $I(G)$ {\bf has linear minimal free resolution} if the minimal free resolution is $k$-steps linear for all $k \geq 1$. 
We also say that $I(G)$ has {\bf linear presentation}
if its minimal free resolution is $1$-step linear. 
\end{definition}
 We proceed in this section by recalling a few well known results. We refer the reader to~\cite{arindam2014}  and~\cite{MR3011341}  for reference.

\begin{observation} 
Let $I(G)$ be the edge ideal of a graph $G$. Then $I(G)^s$ has linear minimal free resolution if and only if $\reg (I(G)^s)=2s$.
\end{observation}


 The following theorem follows from Lemma $2.10$ of \cite{arindam2014}.

\begin{lemma}\label{exact}
Let $I \subseteq S$ be a monomial ideal, and let $m$ be a monomial of degree $d$.  Then
\[
\reg(I) \leq \max\{ \reg (I : m) + d, \reg (I,m)\}. 
\]
Moreover, if $m$ is a variable $x$ appearing in I, then $\reg(I)$ is {\it equal} to one of
these terms.
\end{lemma}
The following theorem due to Fr\"oberg (see Theorem $1$ of~\cite{MR1171260}, and Theorem $1.1$ of~\cite{MR3011341}) is used repeatedly throughout this paper:

\begin{theorem}[\cite{MR1171260}, Theorem $1$]\label{theorem:froberg}
The minimal free resolution of $I(G)$ is linear if and only if the complement graph $G^c$ is chordal, that is no induced cycle in $G^c$ has length greater than three.
\end{theorem}
We also have the following results  due to Francisco, H{\`a} and Van Tuyl.
\begin{theorem}[\cite{MR3011341}, Proposition $1.3$]\label{Adam:1}
 Let $G$ be a graph and $I(G)$ be its edge ideal. Then $I(G)$ has a linear presentation if and only if $G^{c}$ has no induced $4$-cycle. 
\end{theorem}
\begin{theorem}[\cite{MR3011341}, Theorem $1.8$]\label{Adam:2}
 If ${I(G)}^s$ has a linear minimal resolution for a graph $G$ and for some $s\geq 1$, then $I(G)$ has a linear presentation. 
\end{theorem}

The following theorem due to Herzog, Hibi and Zheng is necessary for our results.
\begin{theorem}[\cite{MR2091479}, Theorem $3.2$]\label{Herzog:1}
 Let $I(G)$ be the edge ideal of a graph $G$. If $I(G)$ has linear minimal free resolution, then so has $I(G)^{s}$ for each $s\geq 2$. 
 \end{theorem}

\begin{definition}
For any graph $G$, we write $\reg(G)$ as shorthand for $\reg(I(G))$.  
\end{definition}



 


The following proposition which we state without. 
\begin{proposition}[\cite{MR3199032}, Theorem $3.1$]\label{prop:myprop}
Let $G$ be a connected bipartite graph. The edge ideal $I(G)$ has regularity $3$ if and only if $G^{c}$ has at least one induced cycle of length $\geq 4$ and $G^{bc}$ does not 
contain 
any induced cycle of length $\geq 6$. 
\end{proposition}
Finally we mention the following theorem from~\cite{arindam2014} without proof. This will be the most important structural tool for our proof. 
This theorem shows that the powers of edge ideals have a 
very special property regarding short exact sequences, which makes the task of finding upper bounds for regularity easier.
\begin{theorem}[\cite{arindam2014}, Theorem $5.2$]\label{theorem:upperbound}
For any finite simple graph $G$ and any $s\geq 1$, let the set of minimal monomial generators of $I(G)^s$ be $\{m_1,\dots,m_k\}$, 
then $$\reg(I(G)^{s+1}) \leq \max \{ \reg (I(G)^{s+1} : m_l)+2s, 1\leq l \leq k, \reg ( I(G)^s)\}.$$ 
\end{theorem}

We next prove a result about bipartite graphs that is useful for our purpose.
 
\begin{proposition}\label{proposition:Ali}
If $G$ is a bipartite graph then its complement does not have any induced cycle of 
length $>4$. In  particular edge ideal of a bipartite graph has linear presentation, if and only if all its powers have linear resolution.
\end{proposition}
\begin{proof}
 Let $G$ be a bipartite graph and $V(G)=X\sqcup Y$ be a bipartition of $G$. 
Note since $G^{c}$ contains the complete graph over $X$ and $Y$, 
then we can say every cycle in $G^{c}$ of length $\geq 5$  has at least three $x$'s or three $y$'s. 
Hence it cannot be induced. The second part follows directly from  Theorem~\ref{theorem:froberg}, Theorem~\ref{Adam:1}
and Theorem~\ref{Herzog:1}. 
\end{proof}

\section{Bounding the regularity: The results}
\medskip

 In this section we give some new bounds on $\reg (I(G)^s)$ for biparite graphs $G$ for which $\text{reg}(I(G))=3$. The main idea is to use Proposition 2.15 
 and Theorem 2.16 and the analysis of the ideal 
 $(I(G)^{s+1}:e_1\dots e_s)$ for an arbitrary $s$-fold product of edges (i.e. for $i \neq j$, $e_i=e_j$ is a possibility) in the spirit of~\cite{arindam2014}. 
 Now any $s$-fold product can be written as a product of $s$ edges in various ways. In this section we fix a presentation and work with respect to that. 
We first mention the following important result proved in~\cite{arindam2014} which says that these ideals are generated in degree two for any graph $G$.
 
\begin{theorem}\label{theorem:arindamgraph}
For any graph $G$ and for any $s$-fold product $e_1\dots  e_s$ of edges in $G$ (with the possibility of $e_i$ being same as $e_j$ as an edge for $i \neq j$), 
the ideal 
$$(I(G)^{s+1} :e_1\dots e_s)$$ 
is generated by monomials of degree two.
\end{theorem}
To analyze the generators of $(I(G)^{s+1} :e_1\dots e_s)$, we recall the notion of $\emph{even-connectedness}$ with respect to $s$-fold products from~\cite{arindam2014}.
 
\begin{definition}
Two vertices $u$ and $v$ ($u$ may be same as $v$) are said to be {\bf even-connected} with respect to an $s$-fold product $e_1\dots e_s$ 
if there is a path $p_0 p_1\dots p_{2k+1}$, $k \geq 1$ in $G$ such that:\\
\begin{enumerate}
 \item $p_0=u, p_{2k+1}=v.$
 \item For all $0\leq l \leq k-1$, $p_{2l+1} p_{2l+2}=e_i$ for some $i$.
 \item For all $i$, $$|\{l\geq 0| p_{2l+1} p_{2l+2} =e_i \} | \leq | \{j |e_j=e_i \} | $$
\item For all $0 \leq r \leq 2k$, $p_r p_{r+1}$ is an edge in $G$.\ 
 \end{enumerate}
If these properties are satisfied then $p_0,\dots,p_{2k+1}$ is said to be an even-connection between $u$ and $v$ with respect to $e_1\dots e_s$.
\end{definition}
We make an observation which follows directly from the definition:
\begin{observation}
If $u,v$ are even connected with respect to $e_1\dots e_s$ then they are even connected with respect to $e_{i_1}....e_{i_t}$ for any $\{1,\dots,s\}\subset \{i_1,\dots,i_t\}$.
\end{observation}

By using the concept of even connection the second author gave a description of $(I(G)^{s+1} : e_1\dots e_s)$ for each $s$-fold product. 

\begin{theorem}[\cite{arindam2014}, Theorem $6.7$]\label{theorem:new}
Every generator $uv$ ($u$ may be equal to $v$) of  $(I(G)^{s+1} : e_1\dots e_s)$ is either an edge of $G$  or even-connected with respect to $e_1\dots e_s$, for $s \geq 1$.
\end{theorem}


As bipartite graphs have no odd cycles, the following result is from the definition of even-connectedness. 
\begin{proposition}\label{col:mycol}
 Let $G$ be a bipartite graph and $s\geq 1$ be an integer. Then for every $s$-fold product $e_1\dots e_s$, $(I(G)^{s+1}:e_1\dots e_s)$ is a quadratic squarefree monomial ideal.
 Moreover the graph $G'$ associated to $(I(G)^{s+1}:e_1\dots e_s)$ is bipartite on the same vertex set and same bipartition as $G$. 
\end{proposition}
\begin{proof}
 Note that from Theorem~\ref{theorem:arindamgraph} we know $(I(G)^{s+1}:e_1\dots e_s)$ is a quadratic ideal. On the other hand since $G$ is a bipartite graph, from 
 Theorem~\ref{theorem:konig} we know $G$ contains no odd cycle. Therefore for every vertex $v\in V(G)$ we can say $v$ is not even connected to itself with respect to $e_1\dots e_s$. 
So $(I(G)^{s+1}:e_1\dots e_s)$ is squarefree. Then a graph can be associated to it, namely $G'$. Now we show $G'$ is also bipartite on $V(G)$ with the same bipartition. 

From Theorem~\ref{theorem:new} we 
have $G\subseteq G'$ and $V(G)=V(G')$ but $E(G)\subseteq E(G')$. Let $V(G)=X\sqcup Y$ be the bipartition. We only need to show that $X$ and $Y$ are independent in $G'$. 

Suppose there is $e=uv\in G'$ such that $u,v\in X$. Since $G$ is bipartite $e\notin G$, using Theorem~\ref{theorem:new} we can say $u$ and $v$ are even connected with respect 
to $e_1\dots e_s$. By the definition of even-connectedness there is 
a path $p_0p_1\dots p_{2k+1}$ in $G$ such that $p_0=u$ and $p_{2k+1}=v$. Now note that since $u\in X$ and $p_0p_1\in G$ and 
$G$ is bipartite we can conclude $p_1\in Y$, since $p_1p_2\in G$ and $p_2\in X$ by repeating this process we can say $p_{2k}\in X$. 

But we know $p_{2k}p_{2k+1}\in G$ and 
$p_{2k},p_{2k+1}\in X$, which contradicts the fact that $G$ is bipartite. 
Then $X$ is independent in $G'$. By using the same method we can show that $Y$ is independent in $G'$.
 
 \end{proof}
The following corollary follows directly from  Theorem~\ref{theorem:froberg}, and Lemma $6.14$ and Lemma $6.15$ 
in~\cite{arindam2014}.
 \begin{corollary}\label{lemma:arindamlem}
Let $G$ be any graph and $e_1,\dots, e_s$ be some edges of $G$ which are not necessarily distinct.  If the minimal free resolution of $I(G)$ is linear, then
$$(I(G)^{s+1}:e_1\dots e_s)$$ 
also has a linear minimal free resolution. 
\end{corollary}

To prove the main result of this section we need the following lemma.
\begin{lemma}\label{lemma:Arindamlemma}
 Let $G$ be a bipartite graph and $I=I(G)$ be its edge ideal. Suppose $e_1\dots e_s$ is an $s$-fold product of edges in $G$ for a positive integer $s$. Then we have 
\begin{eqnarray*}
(I^{s+1}:e_1\dots e_s)=\left(\left(I^2:e_i\right)^s:\prod_{j\neq i} e_j\right)&\text{for each $i\in\{1,2,\dots,s\}$}. 
\end{eqnarray*}

\end{lemma}
\begin{proof}
Without loss of generality we can suppose $i=1$.  

Suppose $V(G)=X\sqcup Y$ is a bipartition of $V(G)$. and assume $uv\in (I^{s+1}:e_1\dots e_s)$. 
From Theorem~\ref{theorem:new}, we have $uv\in I$ or $u$ and $v$ are even-connected with respect to $e_1\dots e_s$.
Without loss of generality we can suppose $u\in X$.
If $uv\in I$ then the statement clearly follows. Let us assume $uv\notin I$ and $u$ and $v$ are even-connected with respect to $e_1\dots e_s$. 

From the definition of even-connection there is a path  
 $P:u=x_0y_1x_1\dots x_ky_{k+1}=v$ in $G$ such that  
 \begin{eqnarray*}
\text{for each $i\in\{1,\dots,k\}$},& y_{i}x_{i}=e_j&\text{for some $j\in\{1,2,\dots,s\}$}. 
 \end{eqnarray*}
We put $G'=G((I^2:e_1))$. Since $G\subset G'$ if there is no $i\in\{1,\dots,k\}$ such that $y_{i}x_{i}=e_1$, then the path $P$ is an even-connection with respect to $e_2\dots e_s$
in $G'$. So from Theorem~\ref{theorem:new} the result follows. Therefore we assume there are  
$\alpha_1,\dots,\alpha_{\ell}\in\{1,\dots,k\}$ such that 
\begin{eqnarray*}
 \alpha_1<\alpha_2<\dots <\alpha_{\ell}&\text{and}&\text{$y_{\alpha_t}x_{\alpha_t}=e_1$ for every $t$}. 
\end{eqnarray*}
For each $t$ since $y_{\alpha_t}x_{\alpha_t}=e_1$, $x_{\alpha_t-1}$ and $y_{\alpha_t+1}$ are even-connected with respect to $e_1$ and therefore 
$x_{\alpha_t-1}y_{\alpha_t+1}\in G'$. In particular we have $x_{\alpha_1-1}$ and $y_{\alpha_{\ell}+1}$ are even-connected with respect to $e_1$. 
Then we have the following even-connection with respect to $e_2\dots e_s$ in $G'$
$$P':u=x_0y_1x_1\dots x_{\alpha_1-1}y_{\alpha_{\ell}+1}x_{\alpha_{\ell}+1}\dots y_{k+1}=v$$
and so $uv \in ((I^2:e_1)^s:e_2\dots e_s)$. The conditions (2), (3) and (4) of even connectedness follow as $P$ is an even connection in $G$.

To show the converse suppose  $uv \in ((I^2:e_1)^s:e_2\dots e_s)$. Then from Theorem~\ref{theorem:new} either $uv\in (I^2:e_1)$ or $u,v$ are even-connected with respect to 
$e_2\dots e_s$ in $G'$. If  $uv\in (I^2:e_1)$, the statement is evident (since then clearly $uv\in (I^{s+1}:e_1\dots e_s)$). So we assume 
$uv \notin (I^2:e_1)$. 
From the definition of even-connectedness there is a path $P:u=x_0y_1x_1\dots y_kx_ky_{k+1}=v$ in $G'$ such that  
 \begin{eqnarray*}
\text{for each $i\in\{1,\dots,k\}$},& y_{i}x_{i}=e_j&\text{for some $j\in\{2,\dots,s\}$}. 
 \end{eqnarray*}
If for each $i$, $x_iy_{i+1}\in G$, then $P$ is an even-connection in $G$ and from Theorem~\ref{theorem:new} the claim is evident. So suppose there exists $i$ such that 
$x_{i}y_{i+1}\in G'\backslash G$. From Theorem~\ref{theorem:new} $x_{i}$ and $y_{i+1}$ are even connected with respect to $e_1$. So by the definition there is a path 
$x_{i}e_1y_{i+1}$ in $G$. Let $e_1=yx$. Therefore if we replace $x_{i}y_{i+1}$ by $x_{i}e_1y_{i+1}$ in $P$ we have the following path in $G$ 
$$P':u=x_0y_1x_1\dots x_{i}yxy_{i+1}\dots y_{k+1}=v.$$
If we have only one copy of $e_1$ in $P'$, then clearly $P'$ is 
an even connection with respect to $e_1\dots e_s$ in $G$ and the theorem will follow. Otherwise assume there exists $i$ and $j$ 
such that $i<j$ and $|j-i|$ maximum such that 
\begin{eqnarray*}
x_{i}yxy_{i+1},x_{j}yxy_{j+1}\in P'.
\end{eqnarray*}
Then $P'$ can be reduced to the following path in $G$
$$P'':u=x_0y_1x_1\dots x_{i}yxy_{j+1}\dots y_{k+1}=v.$$
We observe that this is an even-connection with respect to 
$e_1\dots e_s$. The conditions (1), (2) and (4) in the definition are satisfied as $P$ is an even connection in $G'$ and condition (3) 
follows from the fact that $P''$ has only one copy of $e_1$ by construction. This proves the converse.  

\end{proof}

\begin{theorem}\label{theorem:main}
 Let $G$ be a bipartite connected graph and $s\geq 1$ be an integer. If $\reg(I(G))=3$, then $\reg((I(G)^{s+1}:e_1\dots e_{s}))\leq 3$ for every $s$-fold product $e_1\dots e_s$.  
\end{theorem}
\begin{proof}
 Let $e_1\dots e_s$ be an $s$-fold product. 
 Proof by induction on $s$. Let $s=1$.
From Proposition~\ref{col:mycol} $(I(G)^{2}:e)$ 
is a quadratic squarefree monomial ideal and its associated 
graph $G'$ is bipartite with the same vertex set and bipartition as $G$. Also note that if $G'^{c}$ contains no induced cycle of length $\geq 4$, 
then from Theorem~\ref{theorem:froberg} and Proposition 2.17 we can say 
$(I(G)^{2}:e)$ has a 
linear resolution and thus 
$$\reg((I(G)^{2}:e))=2.$$
Then we may assume $G'^{c}$ has an induced cycle of length $\geq 4$. From Proposition~\ref{prop:myprop} we only need to show that there is no cycle of length $\geq 6$ in $G'^{bc}$.

Let $V(G)=X\sqcup Y$ be a bipartition of $G$ and assume there is an induced cycle $C_{2n}$ in $G'^{bc}$ ($n\geq 3$) on the following vertex set 
$$V(C)=\{x_1,\dots,x_n\}\cup\{y_1,\dots,y_n\}.$$ 
\begin{center}
 \includegraphics[width=2 in]{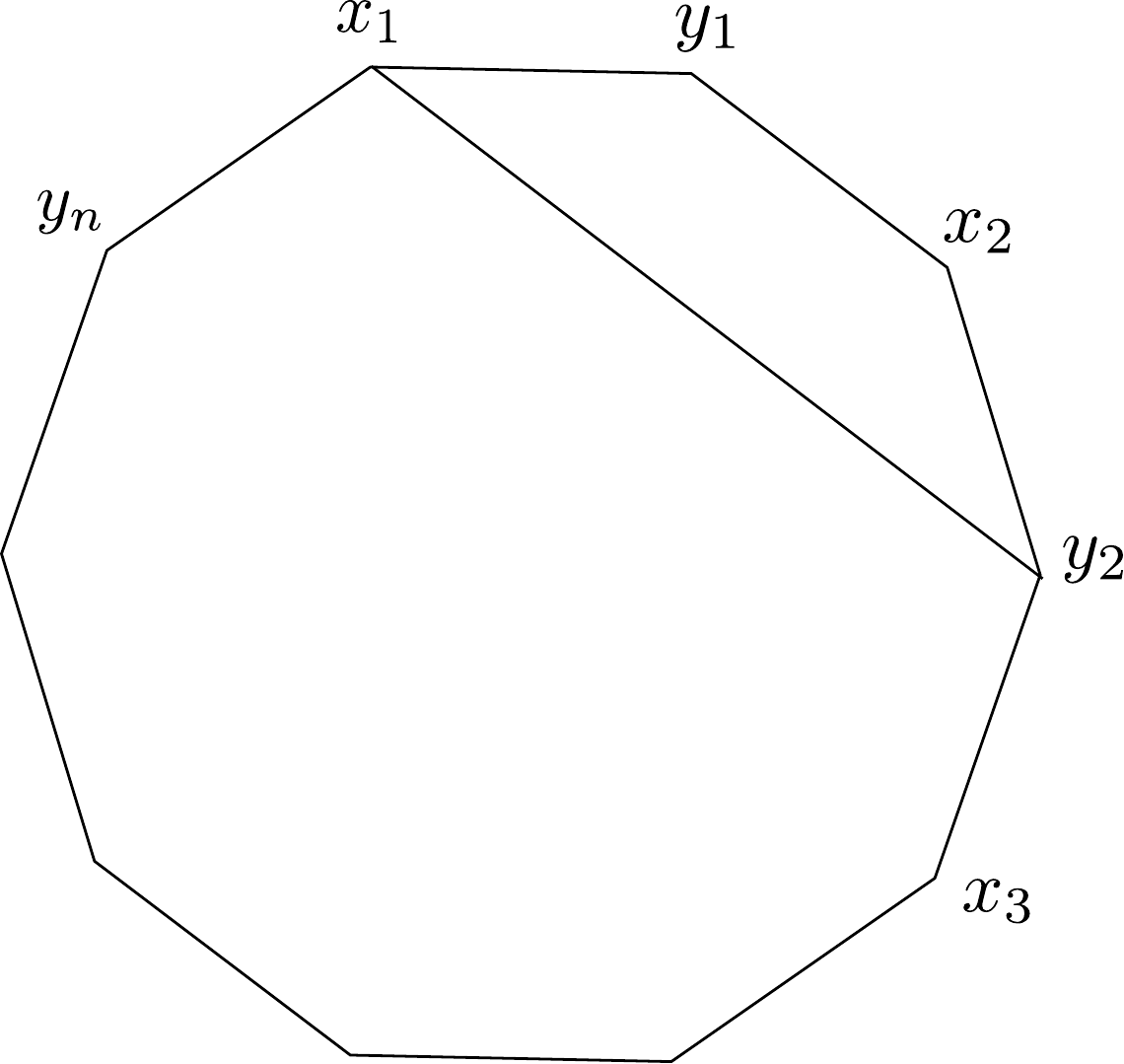}
 \end{center}
Let $C_{2n}:x_1y_1,x_2y_2,\dots,x_ny_n$. Since $G\subseteq G'$,
then $C_{2n}$ is also a cycle in $G^{bc}$. Since $\reg(I(G))=3$, and the length of $C_{2n}$ $\geq$ $6$ from Proposition~\ref{prop:myprop} $C_{2n}$ has to contain a chord.
We can assume this chord divide $C_{2n}$ to two cycles of smaller length. If one of these cycles is $C_4$ we stop, otherwise since its length is $\geq$ $6$ it must have a chord 
(Proposition~\ref{prop:myprop}). 
We get that chord. Again this chord divides $C_{2n}$ into two cycles of smaller length. If one of these cycles is $C_4$ we stop, otherwise 
we keep finding chord to end up with $C_4$ in $G^{bc}$. Without loss of generality we can assume $C_4=x_1,y_1,x_2,y_2$.



First note that by applying Theorem~\ref{theorem:new} because $x_1y_2\in G'\backslash G$,  we can conclude $x_1y_2$ is even connected with respect to $e=xy$. 
Then by the definition of even-connectedness there is a path $x_1yxy_2$ in $G$. 
Using proof by contradiction we will show that the cycle $C_{2n}$ contains an induced cycle of length $\geq 6$ in $G^{bc}$.

We first prove the following useful statements
\renewcommand{\theenumi}{\Roman{enumi}}
\begin{enumerate}
 \item $x_2y,x_3y\notin G$. 
 
If $x_2y\in G$ (or $x_3y\in G$), 
then since $xy_2\in G$ we have the even-connection $x_2yxy_2$ (or $x_3yxy_2$) with respect to $e$ in $G$. 
So from Theorem~\ref{theorem:new} 
$x_2y_2\in G'$ (or $x_3y_2\in G'$); a contradiction.

\item $x_3y_1,x_2y_3\in G$. 

Suppose $x_3y_1\in G^{bc}$ (or $x_2y_3\in G^{bc}$). Then 
from Theorem~\ref{theorem:new} $x_3$ and $y_1$ (or $x_2$ and $y_3$) are even connected with respect to $e$ and then from the definition 
$xy_1\in G$ (or $x_2y\in G$). 
Since $xy_2\in G$ and $x_1y\in G$ we have the even-connections $xy_1x_1y$ (or $xy_2x_2y$) in $G$. Then by applying Theorem~\ref{theorem:new} 
we can conclude $x_1y_1$ (or $x_2y_2\in G'$); a contradiction. 
\end{enumerate}

We now settle our claim for $n=3$. 

Note that $y\neq y_1,y_3$ 
(if $y=y_1$ then $x_1y$ is an edge in $G$ as $x_1y_1$ is, and if $y=y_3$ then $x_2 y$ is an edge in $G$ as $x_2y_3$ 
is an edge in $G$ by assumption, that forces $x_2y_2$ to be an edge in $G'$; both lead to a contradiction). 
Therefore, we can consider the $6$-cycle $x_1y_1x_2yx_3y_3$ in $G^{bc}$. From $(I),(II)$ we know this cycle has no chords in $G^{bc}$, 
thus it is an induced cycle in $G^{bc}$ of length $6$, which contradicts Proposition~\ref{prop:myprop}  and the fact that $\reg(G)=3$.

We now assume $n>3$.

We show the following statements

\begin{itemize}
 \item For each $i\geq 2$, if $xy_i\in G$ then we have 
 \begin{eqnarray}
  x_{i+1}y_j\in G &\text{for each $j\notin\{i,i+1\}$ and}&x_{i}y_j\in G\hspace{.05 in}\text{for each $j\notin\{i,i-1\}$}.\label{equ:1}
 \end{eqnarray}
\begin{proof}
Assume for some $j$ we have $x_{i+1}y_j\in G^{bc}$ (or $x_{i}y_j\in G^{bc}$). 
Then we have an even-connection $x_{i+1}yxy_j$ (or $x_iyxy_j$) that is $x_{i+1}$ (or $x_i$) is connected to $y$ in $G$.
Also since $xy_i\in G$, these  even-connections can be converted to the even connection 
$x_{i+1}yxy_i$ (or $x_iyxy_i$), which from Theorem~\ref{theorem:new} means $x_{i+1}y_i$ or $x_{i}y_{i}$ 
belong to $G'$; a contradiction. 
\end{proof}
\item For each $i\notin\{1,n\}$ 
\begin{eqnarray}
x_iy_n\notin G^{bc}.\label{equ:2} 
\end{eqnarray}

\begin{proof}
Suppose for some $i$, $x_iy_n\in G^{bc}$.
Since $x_iy_n\in G'$, using Theorem~\ref{theorem:new} we have a path $x_iyxy_n$ in $G$. But since $x_1y\in G$ we can conclude
that $x_1$ and $y_n$ are even connected with respect to $e$. So from Theorem~\ref{theorem:new} we have $x_1y_n\in G'$, a contradiction. 
\end{proof}
\end{itemize}
We proceed  by showing the following  
\begin{eqnarray}
 xy_{\ell},x_1y_{\ell}\in G&\text{for all $3\leq \ell\leq n-1$.}\label{eqn:myequ}
\end{eqnarray}
We prove this by using induction on $\ell$. First we assume $\ell=3$. 

If $x_1y_3\in G^{bc}$, since $y\neq y_1,y_3$ 
(if $y=y_1$ then $x_1y$ is an edge in $G$ as $x_1y_1$ is, 
and if $y=y_3$, then $x_2 y$ is an edge in $G$ as $x_2y_3$ is an edge in $G$ by assumption, 
that forces $x_2y_2$ to be an edge in $G'$; both lead to a contradiction) and from (I) and (II) we can consider the induced $6$-cycle $x_1y_1x_2yx_3y_3$ in $G^{bc}$, 
which contradicts the fact that $\reg(G)=3$. Then $x_1y_3\notin G^{bc}$.

If $xy_3\in G^{bc}$, then since $x\neq x_1,x_3$ (because otherwise since $xy_2\in G$ we have $x_1y_2\in G$ or $x_3y_2\in G$)
we can consider the $6$-cycle $x_1y_2x_3y_3xy_n$ in $G^{bc}$. 
From the fact that $x_1y_3\notin G^{bc}$ and (\ref{equ:2}) we know this cycle has no chords, contradicting the fact that $\reg(G)=3$. 

We now suppose for each $3\leq \ell <n$ that our claim is true. We show 
$xy_{\ell+1}, x_1y_{\ell+1}\in G$.

Suppose $x_1y_{\ell+1} \in G^{bc}$. Note that from the induction hypothesis and (\ref{eqn:newequ}) and (\ref{equ:2})  we have 
\begin{eqnarray*}
  x_{t}y_{j}\notin G^{bc} &\text{for $3\leq t \leq \ell < n$ and $j\notin\{t,t+1\}$}\\
 x_1y_{t}\notin G^{bc}&\text{for $3\leq t\leq \ell$}.
 \end{eqnarray*}

Then the cycle 
$C_{2\ell}:x_1y_2x_3y_3\dots x_{\ell+1}y_{\ell+1}$ is an induced cycle of length $\geq 6$ in $G^{bc}$ which contradicts the fact that 
$\reg(G)=3$. Then we have 
\begin{eqnarray}
x_1y_{t}\notin G^{bc}&\text{for $3\leq t\leq \ell+1$.}\label{eqn:newequ}
\end{eqnarray}
We show that $xy_{\ell+1}\in G$. Suppose $xy_{\ell+1}\in G^{bc}$. 

Note that $x\not\in\{x_1,\dots,x_{\ell+1}\}$. Otherwise, using that by the induction hypothesis
we have $xy_t\in G$ for each $t\in\{1,2,\dots, \ell\}$, we can say $x_{j+1}y_j\in G$ for some $j\in\{1,2\dots,\ell\}$ which is a contradiction. 

Then we can consider the $2\ell+1$-cycle
$x_1y_2x_3y_3\dots y_{\ell+1}xy_n $.
By applying the induction hypothesis, (\ref{equ:1}), (\ref{equ:2}) and (\ref{eqn:newequ}) we can conclude this cycle has no chords, contradicting
the fact that $\reg(G)=3$. This settles (\ref{eqn:myequ}).

Then by using (\ref{equ:1}), (\ref{equ:2}) and (\ref{eqn:myequ}), 
we can find the $2(n-1)$ induced cycle $x_1y_2x_3y_3\dots x_n y_n$ in  $G^{bc}$, contradicting Proposition~\ref{prop:myprop} 
and the fact that $\reg(G)=3$.

Now suppose $s>1$ and our claim holds for each $t<s$. Then by Lemma 3.7 and induction we have $\reg(G')\leq 3$. 

If $\reg(G')=2$, then from Corollary~\ref{lemma:arindamlem} and Lemma~\ref{lemma:Arindamlemma}
$$\reg({I(G)}^{s+1}:e_1\dots e_s)=\reg({I(G')}^{s}:e_2\dots e_s)=2.$$

Also if $\reg(G')=3$, from the induction hypothesis and Lemma~\ref{lemma:Arindamlemma} we have 
$$\reg({I(G)}^{s+1}:e_1\dots e_s)=\reg({I(G')}^{s}:e_2\dots e_s)\leq 3.$$

\end{proof}
\begin{theorem}
 Let $G$ be a bipartite connected graph with edge ideal $I(G)$. If $\reg(G)=3$, then for all $s\geq 1$, $\reg(I(G)^{s})=2s+1$. 
\end{theorem}

\begin{proof}
The proof is by induction on $s$. The case $s=1$ is true by assumption. Suppose $\reg(I(G)^s)=2s+1$, we will show that $\reg(I(G)^{s+1})=2s+3$. 

Note that from Theorem~\ref{theorem:main}
and Theorem~\ref{theorem:upperbound} we have 
$$\reg(I(G))^{s+1}\leq 2s+3.$$ 
On the other hand since $\reg(I(G))^{s+1}\geq 2s+2$, if 
$\reg(I(G))^{s+1}<2s+3$ then we have $\reg(I(G))^{s+1}=2s+2$ or in other words $\reg(I(G))^{s+1}$ has a linear minimal free resolution.

By applying Theorem~\ref{Adam:2} we can conclude $I(G)$ has a linear presentation, and since $G$ is bipartite 
from Proposition~\ref{proposition:Ali} it has a linear minimal free resolution. So $\reg(G)=2$; a contradiction. 
\end{proof}

\textbf{Acknowledgements.} 
This paper was prepared when the first author was visiting the University of Virginia.
The first author is very thankful to Professor C. Huneke for his valuable support and the Department of Mathematics 
for their hospitality. He also would like to thank his advisor S. Faridi. The second author is very grateful to his advisor C. Huneke for constant support, valuable ideas and suggestions throughout the project.


\end{document}